\UseRawInputEncoding
\documentclass[12pt]{amsart}
\usepackage{graphicx}
\usepackage{tikz}        
\usepackage{amssymb}
\usepackage{epstopdf}
\usepackage[mathscr]{eucal}
\usepackage{MnSymbol}
\usepackage{amsmath,amssymb,amscd}
\usepackage{color}
\usepackage{epsfig}
\usepackage{bbold}
\usepackage{hyperref}
\newtheorem{theorem}{Theorem}
\newtheorem{lemma}{Lemma}

\newtheorem{definition}{Definition}
\newtheorem{corollary}{Corollary}

\newtheorem{proposition}{Proposition}
\theoremstyle{definition}
\newtheorem{remark}{Remark}

\newtheorem{example}{Example}

\def \mb{\mathbb}

\def \bf{\mathbf}

\def \R{\mb R}                 

\def \a{\alpha}         
\def \b{\beta}           
\def \D{\Delta}         
\def \vp{\varphi}       
\def \th{\theta}       

\newcommand {\rank} {\text{rank}}
\newcommand {\diag} {\text{diag}}

\def \S{\mb S}        
\def\v{{\bf v}}

\def\m{{\bf m}}
\def\x{{\bf x}}

\newcommand {\q} {\mathbf{q}}
\newcommand {\p} {\mathbf{p}}

\def \and{\mbox{and}}
\usepackage[top=4cm, bottom=4cm, left=3.5cm, right=3.5cm]{geometry}
\title{Dziobek equilibrium configurations  on a Sphere }
\begin{document}
\maketitle
\markboth{Shuqiang Zhu}{Dziobek equilibrium configurations  on a sphere}
\author{\begin{center}
 {\textbf Shuqiang Zhu  \footnote{Supported by NSFC(No.11801537).}    }\\
\bigskip
{\footnotesize 
School of Economic and Mathematics,  Southwestern University of Finance and Economics, \\
Chengdu 611130, China\\
 zhusq@swufe.edu.cn 
}
\end{center}

 \begin{abstract}
 We  investigate  the n-body problem  on a sphere with a general interaction potential that depends on the mutual distances. We focus on the equilibrium configurations, especially on the Dziobek equilibrium configurations, which is an analogy of Dziobek central configurations of the classical  n-body problem. We obtain a criterion  and then  reduce it to two sets of equations. Then we apply  these equations to the curved n-body problem in $\S^3$.  In the end, we find the derivative of the  Cayley-Menger determinant. 

 \end{abstract}  
     \vspace{2mm}
   
           \vspace{8mm}
\section{introduction}
 
 The classical n-body problem has been generalized  in many   ways, for example, under  the potential  $\sum \frac{m_im_j}{r_{ij}^\alpha}$, or in higher dimensional Euclidean space. In particular, the \emph{curved n-body problem}, which generalizes the classical  n-body problem to surfaces of constant curvature has received lot of attentions in the last decade (cf \cite{BMK04, Dia13-1, Lim98, Sto18} and the references therein ).
 
 Motivated by those work, we study the generalization of the n-body problem to unit sphere of the Euclidean space. We assume that the potential depends on the shortest geodesic distance. We also assume that the potential is attractive (repulsive) in most  cases. We only specify the potential in the last section. 
 
  One major distinction between the generalization and the classical n-body problem is the existence of equilibrium configurations, due to 
 the compactness of  spheres.  This paper is devoted to the study of equilibrium configurations on spheres.

 In particular, we consider the equilibrium configurations formed by $N$ bodies on some $(N-2)$-dimensional sphere. We call them the \emph{Dziobek equilibrium configurations}.  In the classical n-body problem, Otto Dziobek \cite{Dzi1900} first introduced a set of equations  for non collinear four-body central configurations on $\R^2$, an approach proved fruitful in the study of four-body central configurations (cf \cite{AFS08,Moe01} and the references therein).

  

    We obtain a  criterion similar to that of Otto Dziobek  for the Dziobek  equilibrium configurations in the n-body problem on a sphere. If  the potential is attractive (repulsive), there is an  obstacle for the equilibrium configurations, namely, the particles could not lie on one hemisphere. By this property, we can further separate the criterion into two sets of equations. 
    One set of the equations can be used to determine the manifold in the configuration space that admits equilibrium configurations, then the other set of equations can be used to determine the corresponding masses.  

  The paper is organized as follows. In Section \ref{sec:basic}, we  discuss  the basic setting of the n-body problem  on a sphere and the equilibrium configurations. 
  In Section \ref{sec:de}, we define the  Dziobek equilibrium configurations and  obtain a  criterion.  Then we separate the criterion into two sets.  
  In Section \ref{sec:cnbp}, we turn to the \emph{curved  n-body problem} in $\S^3$. We apply the criterion to equilibrium configurations  of three- four- and five-body in $\S^3$ and discuss the stability of associated equilibria. We discuss the derivative of the Cayley-Menger determinant in the Appendix.

\section{the equilibrium configurations for mechanical system on a  sphere  }\label{sec:basic}
Let $\S^n$ be the unit sphere of the Euclidean space $\R^{n+1}$.  Let us consider $N$ points  of positive mass
$m_i$ on $\S^n$ that interacting mutually by a potential depending on the shortest geodesic distance between the points. The position vector of the $i$-th point is  $\q_i=(x_{i1}, ..., x_{i, n+1})^T$ with $x_{i1}^2+ ... + x_{i, n+1}^2=1$, $i=1, ..., N$. Denote  the configuration by $\q=(\q_1,..., \q_N)$. The configuration space is 
\[   Q= (\S^n)^N\setminus \{ {\rm collisons \ and \  configurations \ where \ the \ vector \ field \ is\ undefined  }  \}.  \] 
The mechanical system is given by the Lagrangian $L: TQ\to \R$
\begin{equation}\label{equ:Lag}  L(\q, \dot \q)= T_\q (\dot \q) -V(\q) \end{equation} 
where $T$ is a Riemannian metric on the configuration space and $V(\q)$ is the interaction potential. Denote the distance between two points $\q_i, \q_j$ by $d_{ij}$. Then $\cos d_{ij} =\q_i \cdot \q_j$.  Assume that the potential $V$ is 
\begin{equation}\label{equ:potential}
V(\q) = \sum_{1\le i<j\le N} m_im_jG(d_{ij}), 
\end{equation} 
where $G:(0,\pi)\to \R$ is some given smooth function. 
\begin{definition} [\cite{Sto18}]
A potential V as given by \eqref{equ:potential}  is  called attractive (repulsive) if the binary potential G, is such that $G'(x) > 0 (G'(x) < 0) $ for all $x \in(0, \pi)$.
\end{definition}

 The \emph{equilibrium motion}, or simply  \emph{equilibrium},  is solution in the form of  $\q(t)=\q(0)$. The configuration $\q(0)$, called a \emph{equilibrium configuration},  is a critical point of $V$. The derivative of $V$ is 
\begin{equation*}
\begin{split}
 \nabla_{\q_i}V(\q)&=\sum_{j=1,j\ne i}^N m_im_j G'(d_{ij})\nabla_{\q_i} d_{ij}=\sum_{j=1,j\ne i}^Nm_im_j G'(d_{ij})\nabla_{\q_i} \cos ^{-1}\q_i\cdot \q_j \\
 &= \sum_{j=1,j\ne i}^N \frac{-m_im_jG'(d_{ij})}{\sin d_{ij}}\nabla_{\q_i} \q_i\cdot \q_j. 
\end{split}
\end{equation*}   
By extending $\q_i \cdot \q_j$ into  a homogeneous function of degree zero in  $\R^{2(n+1)}\setminus\{0\}$, i.e.,  $\frac{\q_i}{\sqrt{\q_i\cdot \q_i}} \cdot \frac{\q_j}{\sqrt{\q_j\cdot \q_j}}$, we obtain 
  \begin{equation*}
 \nabla_{\q_i}V(\q)=\sum_{j=1,j\ne i}^N\frac{-m_im_jG'(d_{ij})}{\sin d_{ij}} [\q_j-\cos d_{ij}\q_i] \ \ i=1,..., N. 
 \end{equation*}
Hence, a configuration $\q\in Q$ is an equilibrium configuration  if $\q$ satisfies the following system 
   \begin{equation}\label{equ:SCCE}
 \sum_{j=1,j\ne i}^N\frac{m_im_jG'(d_{ij})}{\sin d_{ij}} [\q_j-\cos d_{ij}\q_i] =0,\  i=1,...,N.
   \end{equation}

\begin{remark}\label{rmk:stability}
	The Lyapunov stability of such equilibrium  is related with the second variation of $V$.  In particular, 
	the  well-known Lagrange-Dirichlet Theorem  says it is stable if the configuration is an isolated minimum. The converses of this theorem is widely discussed
	 (cf. \cite{Pala2020, Ure2020} and the references therein). If the potential is analytic, then it is unstable if it is not a minimum. 
The equilibrium configurations also lead to relative equilibria of the system \cite{Zhu19-1}. 
\end{remark}

\begin{proposition}
 The $i$-th equation of system \eqref{equ:SCCE} holds if and only if  there is a constant $\th_i$ such that 
 \begin{equation} \label{equ:SCCE1}
  \sum _{j\ne i, j=1}^N \frac{m_jm_i G'(d_{ij})\q_j}{\sin d_{ij}}+ \th_i\q_i=0.
  \end{equation}
 \end{proposition}
\begin{proof}
Assume that  equation \eqref{equ:SCCE1}  holds.  Multiply $\q_i$ to the both sides of equation \eqref{equ:SCCE1} .  Since $\q_i\cdot \q_j =  \cos d_{ij}$ and $\q_i\cdot \q_i = 1$,  we obtain $\th_i =- \sum _{j\ne i, j=1}^N \frac{m_jm_iG'(d_{ij}) \cos d_{ij}}{\sin d_{ij}}.$ Thus equation \eqref{equ:SCCE1}  is equivalent to the $i$-th equation of \eqref{equ:SCCE}. 
\end{proof}

The following result   generalizes one result of    Diacu  \cite{ Dia13-1} for the curved n-body problem, see Section \ref{sec:cnbp}. 
 
 \begin{theorem}\label{thm:nec_scc3}
 Assume that the potential is attractive (repulsive).   There is no equilibrium configuration for any positive masses  in any closed hemisphere of $\S^n$ (i.e. a hemisphere that contains its boundary ), as long as at least one body does not lie on the boundary.  
   \end{theorem}  
   
  \begin{proof}
  Let $\q$ be a  configuration that lies in a closed hemisphere of $\S^n$ and that there is at least one body not  on the boundary. Then there is some point $\v\in \S^n$ such that $\v \cdot \q_i \ge 0$ for all $i$ and at least one of them is strictly positive. Assume that $\v \cdot \q_1$ is the smallest. Then 
  $\nabla_{\q_1} V=0$ implies 
  \[   \sum_{j\ne 1}^N\frac{m_jG'(d_{1j})}{\sin d_{1j}} [\v \cdot \q_j-\cos d_{1j}\v\cdot \q_1] =0.   \]
  Since we have assumed $G'(d_{1j})$ is of the same sign for all $j$,  this is a contradiction.
  \end{proof}

We end this  section by several  examples of equilibrium configurations for equal masses.  The examples extend those constructed by  Diacu  in \cite{ Dia13-1} for the 
 curved n-body problem (Section \ref{sec:cnbp}).  Denote the standard bases of $\R^{n+1}$ by $e_1, ..., e_{n+1}$. Denote  the unit sphere in $span\{ e_1, ..., e_{k+1}\}$ by $\S^k$. We assume that the configurations constructed below are not those where $G'(d_{ij})$ is undefined. 
  
   \begin{example}[regular simplex with equal masses ]\label{exm:regular-simplex} 
   	Consider a regular $k$-simplex.  Place one  unit mass at each of the vertices.
   	The configuration obtained is an equilibrium configuration. It is enough to    check that 
   	 equation \eqref{equ:SCCE1}  holds for $i=1$.  Since
$d_{ij}=d_{12}$ for any pair of $\{i,  j\}$, we find that  
$$\sum _{i=2}^{k+1} \frac{m_i G'(d_{i1}) \q_i}{\sin d_{i1}}=\frac{ G'(d_{21})}{\sin d_{21}} (\sum _{i=2}^{k+1} \q_i)= -\frac{ G'(d_{21})}{\sin d_{21}} \q_1. $$ 
  \end{example}

  \begin{example} [regular polygon with equal masses]
  		Consider a regular $2k+1$-gon located on the unit circle of $span\{ e_1, e_2\}$.  Place one  unit mass at each of the vertices. By complex number notation, the position vectors are 
  		$\q_j = e^{i\phi j}$, $j=1, ..., 2k+1$, $\phi=\frac{2\pi}{2k+1}$. let us check equation \eqref{equ:SCCE1}   for $j=2k+1$. Since  $d_{2k+1, j}=d_{2k+1, 2k+1-j}$, we have 
  		\begin{align*}
 &\frac{m_j  G'(d_{2k+1,j}) e^{i\phi j}}{\sin d_{2k+1, j}}   +\frac{m_{2k+1-j} G'(d_{2k+1,2k+1-j}) e^{i\phi (2k+i-j)}}{\sin d_{2k+1, 2k+1-j}}   \\
&=2 \frac{G'(d_{2k+1,j}) }{\sin d_{2k+1, j}}  \cos j \phi e^{i\phi (2k+1)}.  
  		\end{align*}
Then it follows  that equation \eqref{equ:SCCE1} holds for $i=2k+1$, then for all $i$ by symmetry. 

Similarly, the regular polygon of even vertices  with equal masses is also an equilibrium configuration. 
  \end{example}

  \begin{example}[two regular polygons with equal  masses on  two complementary circles]
  	Consider one regular $n_1$-polygon located on the unit circle of   $span\{ e_1, e_2\}$ and another regular $n_2$-polygon located on the unit circle of  $span\{ e_3, e_4\}$.  Place a unit mass at each of the vertices.  Note that the distance between the particles from different polygons is always  $\frac{\pi}{2}$ since 
  	\[  \cos d_{ij}=  \q_i \cdot \q_j =0, \ \  1\le i\le n_1, \ n_1+1\le j\le n_1+n_2.  \]
  	Let us check 
  	that  equation \eqref{equ:SCCE1}  holds for any $1\le i\le n_1+n_2$, say $i=1$. 
  	\begin{align*}
\sum _{i=2}^{n_1+n_2} \frac{m_i G'(d_{i1})\q_i}{\sin^3 d_{i1}}=\sum _{i=2}^{n_1}\frac{ G'(d_{i1})}{\sin d_{i1}}\q_i+G'(\frac{\pi}{2}) \sum _{i=n_1+1}^{n_1+n_2}\q_i.
  	\end{align*}
  	The first part is  collinear with $\q_1$ by the above example,  and the second part is zero. Thus this configuration is one equilibrium configuration. 
  \end{example}

 \section{Dziobek Equilibrium Configurations }\label{sec:de}
  In this section, we consider equilibrium configurations   where $N$ masses span an $(N-2)$-sphere. We obtain a criterion, then separate it into  two sets of equations, the shape equations and the mass equations.
   In the classical n-body problem, a central  configuration of $N$ bodies that span an  $(N-2)$-dimensional  affine plane are called \emph{Dziobek central configurations} \cite{Dzi1900, Moe01}.  For equilibrium configurations on sphere, 
 equation \eqref{equ:SCCE1}  implies that the $N$ position vectors are always dependent, so  $1\le \rank (\q_1, ...,\q_N) \le N-1$.  
  
  \begin{definition}
  A Dziobek configuration of $N$  bodies on sphere  is one such that $ \rank (\q_1, ...,\q_N) = N-1$. 
  \end{definition}

  Let $\{ \q_1, \cdots, \q_N\}$ be a collection of vectors in $\R^{N-1}$. Assume the rank of these $N$ vectors is $N-1$. Consider the $(N-1)\times N$ matrix:
  \[ X=[\q_1,\cdots, \q_N ].  \] 
  Since the rank of $X$ is $N-1$,  $\dim \ker X=1$. The kernel can be found as follows.  Let $X_k$ be the $(N-1)\times (N-1)$ matrix obtained from $X$ by deleting the $k$-th column and let $|X_k|$ denote its determinant. 
  \begin{lemma}\label{lemma:ker}
  Let 
  \begin{equation}\label{equ:D}
  \D=(\D_1,...\D_N) =(  |X_1|, -|X_2|, ..., (-1)^{k+1}|X_k|,...).  
  \end{equation}
  Then $\D^T$ is the base of $\ker X$. In other words, $\D\ne 0$ and $\D_1 \q_1 +\cdots + \D_N \q_N=0$. 
  \end{lemma}
  
  \begin{proof}
  Assume that $\D_N=(-1)^{N+1} |X_N|\ne 0$. Consider the linear system in  $X_N u = \q_N$, $u=(u_1, ..., u_{N-1})^T$.  By Crammer's rule, we obtain
  $ u_k=  \frac{-\D_{k}}{\D_N}, k=1, ..., N-1$. 
  Then it  follows that $\D_1 \q_1 +\cdots + \D_N \q_N=0$. 
  \end{proof}
  
  \begin{proposition}\label{prop:Dine0}
  	Consider a Dziobek configuration of $N$ bodies on $\S^{N-2}$. Then the configuration is not on a hemisphere if and only if all $\D_i$ are of the same sign. 
  \end{proposition} 
  \begin{proof}
  	We only prove that if not all $\D_i$ are of the same sign the Dziobek configuration lies on a hemisphere. There are two cases. 
  	
  	If there is some $\D_i=0$, say $\D_1$, then $\rank\{ \q_2, ..., \q_N\}=N-2$. Let $\Pi$ be the hyperplane spanned by $\{ \q_2, ..., \q_N\}$ and $\vec{n}$ be the normal of $\Pi$ in $\R^{N-1}$ with the property $\vec{n}\cdot \q_1>0$. Then 
  	we have 
  	\[ \vec{n} \cdot \q_i\ge 0, \ i=1, ..., N,  \]
  	which implies that the Dziobek configuration lies on a hemisphere. 
  	
  	If all $\D_i$ are nonzero, there are two consecutive elements of $\D$ that are of different sign, say $\D_1>0, \D_2<0$. Then 
  	\[ |X_1|=\det(\q_2, \q_3, ..., \q_N)>0, \  |X_2|=\det(\q_1, \q_3, ..., \q_N)>0. \]
  	Let $\tilde \Pi$ be the $(N-2)$-dimensional hyperplane spanned by $\{ \q_3, ..., \q_N\}$ and $\vec{m}$ be the normal of $\tilde \Pi$ in $\R^{N-1}$ with the property $\vec{m}\cdot \q_1>0$. Assume that $\q_2=\lambda_1 \q_1 +\sum_{i=3}^N \lambda_i \q_i$. Then 
  	\[  |X_1|= \det(\lambda_1 \q_1 +\sum_{i=3}^N \lambda_i \q_i, \q_3, ..., \q_N )= \lambda_1 |X_2|.    \]
  	Then $\lambda_1>0$. Hence we have 
  	\[ \vec{m} \cdot \q_i\ge 0, \ i=1, ..., N,  \]
  	which implies that the Dziobek configuration lies on a hemisphere. 
  \end{proof}

  Denote  the quantity $\frac{G'(d_{ij})}{\sin d_{ij}}$ by $S_{ij}$. Then equation \eqref{equ:SCCE1}  becomes 
  \[  \sum_{j\ne i} m_j S_{ij} \q_j+ \th_i \q_i =0, \ 1\le i\le N. \]
 
  \begin{theorem} \label{thm:cri_sccn2}
  Assume that the potential is attractive (repulsive) and that  $\q=(\q_1, ..., \q_{N})$ is  a Dziobek configuration  in $\S^{N-2}$. Then the  configuration  $\q$ is an equilibrium configuration  if and only if there is a  nonzero real  number $p$ such that  
  \begin{equation} \label{equ:SCCE2}
  m_im_j S_{ij}=p \D_i \D_j \ {\rm for \ any\ } i\ne j.  
  \end{equation}
  \end{theorem}
  
  \begin{proof}
  The proof of the sufficient  conditions: 
  Since $\q$ is an equilibrium configuration,  equation \eqref{equ:SCCE1} holds. That is,  there is some nonzero real number $p_j$ such that
  \begin{equation}\label{equ:SCCE3}
(m_1 S_{1j}, m_2S_{2j},..., \th_j, ..., m_N S_{Nj}  )=p_j (\D_1, ..., \D_N)\ {\rm for\ all}\ j=1,..., N. 
  \end{equation}
  by Lemma \ref{lemma:ker}.  System \eqref{equ:SCCE3} is equivalent to 
  \[   S=\begin{bmatrix}
  p_1\\p_2\\\vdots\\p_N
  \end{bmatrix} ( \frac{\D_1}{m_1}, ..., \frac{\D_N}{m_N}), \ \    \   {\rm where }  \  S=\begin{bmatrix} 
  \th_1& S_{12}&\cdots&S_{1N}\\
  S_{21} & \th_2&\cdots&S_{2N}\\
  \vdots&\vdots&\ddots&\vdots\\
  S_{N1}&S_{N2}&\cdots&\th_N
  \end{bmatrix}  \]
 Since the left matrix $S$ is symmetric, we see that $p_j\frac{\D_1}{m_1}=p_1\frac{\D_j}{m_j}$, or,   by Proposition \ref{prop:Dine0}, 
  \[  (p_1m_1, ..., p_N m_N )= \frac{p_1m_1}{\D_1} (\D_1, ..., \D_N).   \]
 Let $M=\diag\{ m_1, ..., m_N\}$. We have 
 \[   MSM=\frac{p_1m_1}{\D_1} (\D_1, ..., \D_N)^T  (\D_1, ..., \D_N),  \]
  which gives \eqref{equ:SCCE2}.
  
The proof of the sufficient  conditions: Let $ (p_1m_1, ..., p_N m_N )=p\D$. The system \eqref{equ:SCCE2} implies system \eqref{equ:SCCE3}, so the condition is also sufficient.  
  \end{proof}

  The system \eqref{equ:SCCE2} can be obtained in another way, see Appendix. It  imply that  all $\D_i$ $(i\ge 1)$ are of the same sign, which agrees with Proposition \ref{prop:Dine0}. Eliminating the constant $p$, we get a system of   $\frac{N(N-1)}{2}-1$ equations from \eqref{equ:SCCE2}.  The system can be written in a form with the property that most of the equations are just constraints on the shapes, or, 
 independent of  the masses.


 \begin{proposition}
	Let $A=(a_{ij})$ be a symmetric matrix and $b=(b_1, ..., b_n)$. Assume that $A=b^T b$ and $b_1b_2 b_n\ne 0$. Consider the system consists of the  $\frac{n(n-1)}{2}$ equations 
	\[  a_{ij}=b_i b_j, \ i=1, ..., n-1, \ j=i+1, ..., n.  \]
	The system of equations is equivalent to 
	\begin{align}
	&a_{1n}=b_1b_n, \  b_k=b_n \frac{a_{1k}}{a_{1n}} \  k=2, ..., n-1; \label{equ:system-11}\\
	&b_1=b_n\frac{a_{12}}{a_{2n}}, \  a_{2k}a_{1n}=a_{1k} a_{2n}, \  k=3, ..., n-1;\label{equ:system-21}\\
	&a_{jn}a_{12}=a_{1j}a_{2n}, \  a_{jk}a_{1n}=a_{1k} a_{jn}, \  j =3, ..., n-2, \ k=j+1, ..., n-1;\label{equ:system-31}\\
	&a_{n-1,n}a_{12}=a_{1n-1}a_{2n}. \label{equ:system-41}
	\end{align} 
\end{proposition}

\begin{proof}
	From the first system, we see 
	$a_{ij}a_{kl}= b_ib_jb_kb_l= a_{ik}a_{jl}$ holds for any 4-tuple $\{i, j, k, l\}$ of $\{1, ..., n\}$. Then we derive the second system from the first one. 
	
	Let us derive the first system from the second one. For convenience, put the first system in an upper triangular form
	\[  E=\begin{bmatrix}
	a_{12}=b_1 b_2& a_{13}=b_1 b_3& \cdots & a_{1n}=b_1 b_n\\
	&a_{23}=b_2 b_3&  \cdots &a_{2n}=b_2 b_n\\
	& \ddots&\vdots&\vdots\\
	&& & a_{n-1,n}=b_{n-1} b_n
	\end{bmatrix} \]
	By \eqref{equ:system-11}, we can recover the first row of $E$. 
	By \eqref{equ:system-21} and the first row of $E$, we see 
	\[  a_{2n}=a_{12}  \frac{b_n}{b_1}=b_2 b_n, a_{2k}=\frac{a_{1k} a_{2n}}{a_{1n}}= \frac{b_kb_2 b_n}{b_n} =b_2b_k, k=3, ..., n-1.   \] 
	Hence the second row of $E$ is recovered. Similarly, the $j$-th row can be obtained by 
	\[ a_{jn}=a_{1j}  \frac{b_n}{b_1}=b_j b_n, a_{2k}=\frac{a_{1k} a_{jn}}{a_{1n}}= \frac{b_kb_j b_n}{b_n} =b_jb_k,  k=j+1, ..., n-1.  \]
	Thus, we obtain all equations of $E$.  This completes the proof.
\end{proof}
Applying the above result to the system \eqref{equ:SCCE2}, where $A=S$ and $b=\sqrt{p}\D M^{-1}$,   we get 
\begin{theorem}\label{thm:final}
 Assume  that the potential is attractive (repulsive) and that  $\q=(\q_1, ..., \q_{N})$ is  a Dziobek configuration  in $\S^{N-2}$ with masses $m_1, ..., m_N$. Then $\q$ is an equilibrium configuration  if and only if the following system of equations are satisfied
 \begin{equation}\label{equ:final}
	\begin{cases}
 m_2 = \frac{ S_{1N}\D_2}{S_{12}\D_N}m_N, \    ...,  m_{N-1} = \frac{S_{1N}\D_{N-1}}{S_{1, N-1}\D_N}m_N, 
m_1= \frac{ S_{2N}\D_1}{S_{12}\D_N}m_N,  \cr
 S_{1j} S_{2N}=S_{12}S_{jN}, j=3, ..., N-1;\cr
 S_{jk} S_{1N}=S_{1k}S_{jN}, k=j+1, ..., N-1, \  j=2, ..., N-2.
\end{cases}
 \end{equation}
\end{theorem}

  Note that 
 the first $N-1$ equations are involved with the masses, while the remaining equations are not. 
 Let us call the first $N-1$ equations the \emph{mass equations}, and the others  $\frac{N(N-3)}{2}$ equations the \emph{shape equations}. 

The system \eqref{equ:final} determines the Dziobek equilibrium configurations, including the configurations and the corresponding masses. The shape equations alone can not determine the configurations. Indeed,  there are configurations that satisfies 
 the shape equations, but the configuration lies on a hemisphere (see Remark \ref{rmk:example}  in Section \ref{sec:cnbp}). Thanks to  Proposition \ref{prop:Dine0}, 
 to build a Dziobek equilibrium configuration,  we  may first   construct a configuration that satisfies the shape equations and the condition that all $\D_i$ are of the same sign (or equivalently, configurations not on a hemisphere),  then determine the corresponding positive masses by the mass equations. 
 

\begin{remark}
 Recall that for Dziobek central configurations of    the classical  n-body problem \cite{Moe01}, one can only be certain that at least two elements of $\D$ are nonzero. Hence, we can get a system similar to \eqref{equ:final} there, which is   only necessary but not sufficient. 
 However, for $n=4$, we know that all $\D_i$ are nonzero, so the central configuration equations can be written in a form  similar to \eqref{equ:final}. The difference is that there are furthere restrictions on the mutual distances such that the masses are positive, see Corbera et al. \cite{CCR19}.  \end{remark}

\section{ Example: the curved N-body problem  in $\S^3$}\label{sec:cnbp}

 In this section we consider the problem in $\S^3$ with the gravitational interaction. The potential is defined as spherical-symmetric solutions of the Laplace equation on $\S^3$. This is the \emph{curved n-body problem} in $\S^3$. For more on this problem, see \cite{BMK04, Dia13-1, Zhu19-1}. For any two points $\q_i$ and $\q_j$, the binary potential and the potential are 
 \[  G(d_{ij})=  -m_i m_j \cot d_{ij}, \ V(\q)= \sum_{1\le i<j\le N} -m_i m_j \cot d_{ij} \ \]
respectively.   Since $G(x)=-\cot x$, $G'(x)=\frac{1}{\sin^2 x}>0$, the potential is attractive and $S_{ij}=\frac{G'(d_{ij})}{\sin d_{ij}}=\frac{1}{\sin^3 d_{ij}}$. 
 
 Note that the potential is undefined at $d_{ij}=\pi$, so we must exclude those configurations  with points diametrically opposite in the examples considered in Section \ref{sec:basic}, for  instance, the regular polygons with even vertices. Moreover, there is no equilibrium configuration  for  two masses. Otherwise, the equation \eqref{equ:SCCE1}  implies that $d_{12}$ is $0$ or $\pi$. The equilibrium configurations are also called \emph{special central configurations} in the curved n-body problem in $\S^3$ \cite{Zhu19-1}.


 \subsection{Criteria for Dziobek Equilibrium Configurations  of Three, Four and Five Bodies}
By  Theorem \ref{thm:final}, we obtain the following   criteria for Dziobek equilibrium configurations  of 3, 4 and 5 bodies respectively.  The regular 2, 3, and 4-simplex with equal masses (see Example \ref{exm:regular-simplex}) satisfies the following criteria respectively. 

  
  \begin{corollary}[$N=3$, $\S^1$] 
  Consider one configuration $\q=(\q_1, \q_2, \q_3)$ on $\S^1$.
  Then  $\q$ is a Dziobek equilibrium configuration if and only if   
  the  masses are 
   \[ (m_1, m_2, m_3)= \left( \frac{\sin^3 d_{12}\D_1}{\sin^3 d_{23}\D_3}, \frac{\sin^3 d_{12}\D_2}{\sin^3 d_{13}\D_3 },  1\right)m_3. \] 
  \end{corollary} 

 \begin{corollary}[$N=4$, $\S^2$] \label{cor: 4DSCC}
	Consider one  configuration $\q=(\q_1, \q_2, \q_3, \q_4)$ on $\S^2$.
	Then $\q$ is a Dziobek equilibrium configuration if and only if  
	$$\sin{d_{12}}\sin{d_{34}}=\sin{d_{13}}\sin{d_{24}}=\sin{d_{14}}\sin{d_{23}}$$
	and   the  masses are 
	\[\ (m_1,m_2,m_3,m_4)= \left( \frac{\sin^3 d_{12}\D_1}{\sin^3 d_{24}\D_4}, \frac{\sin^3 d_{12}\D_2}{\sin^3 d_{14}\D_4 }, \frac{\sin^3 d_{13}\D_3}{\sin^3 d_{14}\D_4 },  1 \right)m_4. \] 
\end{corollary}

\begin{corollary} [$N=5$, $\S^3$] \label{cor: 5DSCC}
	
	Consider  one configuration $\q=(\q_1, \q_2, \q_3, \q_4,\q_5)$ in $\S^3$.
	Then $\q$ is a Dziobek equilibrium configuration if and only if  
	\begin{align*}
	\sin{d_{13}}\sin{d_{25}}&=\sin{d_{12}}\sin{d_{35}}=\sin{d_{23}}\sin{d_{15}}, &\sin{d_{34}}\sin{d_{15}}=\sin{d_{14}}\sin{d_{35}},\\ \sin{d_{14}}\sin{d_{25}}&=\sin{d_{12}}\sin{d_{45}}
	=\sin{d_{24}}\sin{d_{15}},
	\end{align*}
	and the  masses are
	\[ (m_1, ..., m_5)= \left( \frac{\sin^3 d_{12}\D_1}{\sin^3 d_{25}\D_5}, \frac{\sin^3 d_{12}\D_2}{\sin^3 d_{15}\D_5 }, \frac{\sin^3 d_{13}\D_3}{\sin^3 d_{15}\D_5 }, \frac{\sin^3 d_{14}\D_4}{\sin^3 d_{15}\D_5 }, 1 \right)m_5. \] 
\end{corollary}

  \begin{figure}[!h]
  	\centering
  \includegraphics*[scale=0.8]{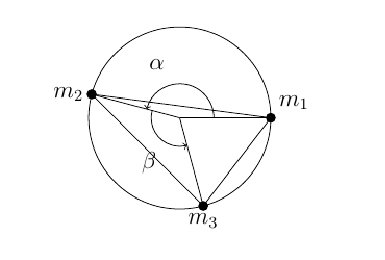}
    \caption{An acute triangle configuration}
    \label{fig:acute}
    \end{figure}
  
  For the case of $N=3$, the constraint on the shape is only that 
the configuration is   not  in one semicircle, in other words, $\vp_{i+1}-\vp_i<\pi$ for $i=1, 2, 3$ with 
 $\q_i=(\cos \vp_i, \sin\vp_i)$. Then all angles are acute, and the configuration  forms  an acute triangle, see Figure \ref{fig:acute}.   Let $d_{12}=\alpha, d_{23}=\b$. Then $d_{13}=2\pi-(\a +\b)$ and \begin{equation}\notag 
    0<\alpha<\pi,\ \ \ 0<\beta<\pi,\ \ \ \pi<\alpha+\beta<2\pi.
    \label{equ:con_scc_3}
    \end{equation} 
Note that  $\D_1=|\q_2, \q_3|=\sin d_{23}= \sin \b$, $\D_2=-|\q_1, \q_3|=\sin d_{13}=|\sin(\a+\b)|$, $\D_3=|\q_1, \q_2|=\sin d_{12}=\sin \a$. Thus the masses satisfy
  \begin{equation} \notag
  \frac{m_2}{\sin^2 \alpha}= \frac{m_3}{\sin^2 (\alpha +\beta)},\ \ \ \ \frac{m_1}{\sin^2 \alpha}= \frac{m_3}{\sin^2 \beta},\ \ \ \ 
     \frac{m_2}{\sin^2 \beta}= \frac{m_1}{\sin^2 (\alpha +\beta)}.
  \end{equation}
  
  The above system gives all Dziobek equilibrium configurations for three masses and it has been obtained by  direct computations in \cite{DSZ16}. 
The constraint of the masses is found as
\[  m_1^2m_2^2+ m_1^2m_3^2+m_2^2m_3^2 -2m_1m_2m_3<0, \] 
if we assume that $\sum_{i=1}^3m_i=1$. It is easy to see that all such configurations are local minima of the potential $V$ restricted on $\S^1$. These equilibria are stable on $\S^1$ (Remark \ref{rmk:stability}), see \cite{DSZ16} and the generalization in \cite{YZ19}.

   For the case of $N=4$,  the system is not trivial and an equivalent system has been obtained by direct computation in \cite{BDZ17}.  We do not know much besides the regular tetrahedron equilibrium configuration on $\S^2$ with  four equal masses. 
  Now we present a family of 4-body Dziobek equilibrium configurations   which contains the regular tetrahedron. 
  Consider a tetrahedron configuration of    four masses with position vectors  
   \begin{align*}
  & \q_1=(1, 0,0)^T\ & \q_2=(-c, r, 0)^T\\
  &\q_3=(-c,   -\frac{1}{2}r, \frac{\sqrt{3}}{2}r)^T \ &\q_4=(-c,   -\frac{1}{2}r, -\frac{\sqrt{3}}{2}r) ^T
     \end{align*}
  where $m_2=m_3=m_4$,  and 
$c\in (0,1)$, $r^2 +c^2 =1$,  see Figure \ref{fig:non-regu-tetra}.   Denote  such a configuration by $\q_c$. 
   \begin{figure}[!h]
         	\centering
     	\includegraphics[scale=0.8]{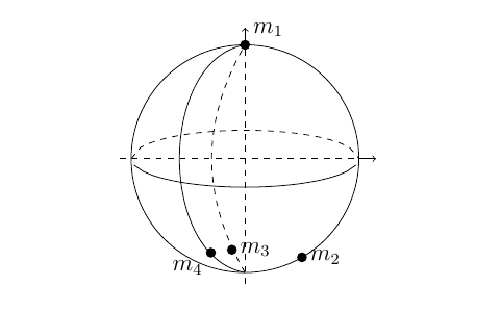}
        \caption{Configuration $\q_c$ on $\S^2$  }
          \label{fig:non-regu-tetra}  
          \end{figure}

  \begin{proposition}\label{pro:4body}
  The configuration $\q_c, c\in (0,1)$ is  a Dziobek equilibrium configuration if 
  \begin{equation}\label{equ:sym_4DSCC}
\frac{m_1}{m_4}= \frac{8\sqrt{3}c}{3(1+3c^2)^{\frac{3}{2}}}
  \end{equation}.
  \end{proposition}  

By numerical study, all such equilibrium configurations are not  minima of the potential $V$ restricted on $\S^2$. These equilibria  are unstable on $\S^2$, see Remark \ref{rmk:stability}.

\begin{proof}
	The tetrahedron is not on one hemisphere and the shape equations are satisfied since 
$d_{12}=d_{13}=d_{14}$, and $d_{23}=d_{24}=d_{34}$. The last two of the mass  equations $ \frac{m_i}{m_4}= \frac{\sin^3 d_{1i}\D_i}{\sin^3 d_{14}\D_4}, i=2,3$ are true since 
  since $\D_2=\D_3=\D_4$.  We only need to check the first mass equation. 
  
   Direct computation leads to 
\[\cos d_{12}=-c, \ \sin^3 d_{12}=r^3, \ \cos d_{24}=c^2 -\frac{1}{2}r^2, \ \sin^3 d_{24}=3\sqrt{3}r^3 ( \frac{1}{4}+\frac{3}{4}c^2) ^{\frac{3}{2}},  \]
and $\D_4= \frac{\sqrt{3}}{2} r^2,  \D_1= 3c \D_4$. Thus the configuration is a Dziobek equilibrium configuration  if and only if 
\[ \frac{m_1}{m_4}=  \frac{\sin^3 d_{12}\D_1}{\sin^3 d_{24}\D_4}=\frac{r^3  3c\D_4}{3\sqrt{3}r^3 ( \frac{1}{4}+\frac{3}{4}c^2) ^{\frac{3}{2}}\D_4}=\frac{8\sqrt{3}c}{3(1+3c^2)^{\frac{3}{2}}}.  \]  
\end{proof}
  \begin{remark}\label{rmk:example}
 	Consider the configuration with  $c=-\frac{1}{3}$. Then $\sin d_{12}=\sin d_{24}$,   so the shape equations are satisfied. However, the  configuration is on the north hemisphere. 
 \end{remark}

 As $c\to 0$, we have $\frac{m_1}{m_4}\to 0$. This is intuitively clear.  As $c\to 0$, the three masses $m_2, m_3, m_4$ tend to form an equilibrium configuration of their own on the equator. Then we may place an infinitesimal mass at $\pm(1,0,0)$ to form an equilibrium configuration of 4 bodies. The function $f(c)=\frac{8\sqrt{3}c}{3(1+3c^2)^{\frac{3}{2}}}, c\in (0,1), $ is increasing on $(0, \frac{\sqrt{6}}{6})$ and decreasing on $( \frac{\sqrt{6}}{6},1)$. The maximum is $\frac{16}{9\sqrt{3}}>1$,  $\lim_{c\to0} f(c)=0$  and $\lim_{c\to1} f(c)=\frac{\sqrt{3}}{3}$.    
  
 \begin{corollary}\label{cro:4body}
 Consider four masses  $(\bar m, m, m, m)$ on $\S^2$.  If $\frac{\bar m}{m}\in (0, \frac{16}{9\sqrt{3}}]$, then there is at least one Dziobek equilibrium configuration.  If $\frac{\bar m}{m}\in (\frac{\sqrt{3}}{3}, \frac{16}{9\sqrt{3}})$, then there are at least two Dziobek equilibrium configurations. Especially, there are at least two equilibrium configurations  for four equal masses. 
 \end{corollary}

 For the case of $N=5$,  the system is not trivial and an equivalent system  has been obtained by direct computation in \cite{BDZ17}.  We do not know much besides  the regular pentatope  equilibrium configuration on $\S^3$ with  five equal masses. Nevertheless, it is easy to construct  a family of 5-body Dziobek equilibrium configurations  similar to the 4-body equilibrium configurations  constructed above and obtain conclusions similar to Proposition \ref{pro:4body} and Corollary \ref{cro:4body}.

\subsection{Another Example}

Consider Dziobek equilibrium configurations of $N$ masses with the property that $\sum_{i=1}^N m_i \q_i =0$.  	By Lemma \ref{lemma:ker}, the vector $(m_1, ..., m_N)$ is a multiple of $(\D_1, ..., \D_N)$. Then equations  of \eqref{equ:SCCE2} implies that  $\sin d_{ij}$  is a constant for  all pairs of  all $\{i, j\}$. Thus, there is some $c\in(0,\pi)$ such that $d_{ij}=c$, or $\pi-c$.

If  all $d_{ij}$ equal to $c$, thus the configuration is a regular simplex, which implies  that $\D_1=\D_2=...=\D_N$ and  $m_1=m_2=...m_N$.  For instance, on $\S^1$,  this is the only possibility. However, this is not the only case if the sphere is of higher dimension. A similar phenomenon happens in \cite{DZ20}. 

For example,  consider the following Dziobek configuration on $\S^2$ with position vectors 
 \begin{align*}
& \q_1=(a, b,0)^T\ & \q_2=(-c, r, 0)^T\\
&\q_3=(-c,   -\frac{1}{2}r, \frac{\sqrt{3}}{2}r)^T \ &\q_4=(-c,   -\frac{1}{2}r, -\frac{\sqrt{3}}{2}r) ^T
\end{align*}
where 
$a, b, c\in (0,1)$, $r^2 +c^2 =1, a^2 +b^2 =1$.  We show that there are values of $a, c$ such that $\sin d_{ij}$ is a constant for all pairs of $\{i,j\}$. Since the configuration is not on one hemisphere, this configuration leads to a Dziobek equilibrium configuration with the property $\sum \m_i \q_i=0$ but not a regular simplex.  
Indeed, we only need to solve 
\[  \q_1\cdot \q_2=\q_2\cdot \q_3, \  \q_1\cdot \q_2=-\q_1\cdot \q_3.  \] 
In coordinates, the system is 
\[  4ac=\sqrt{(1-a^2)(1-c^2)}, \ 3c^2-1=6ac.  \]
\begin{figure}[!h]
	\centering
	\includegraphics[scale=0.4]{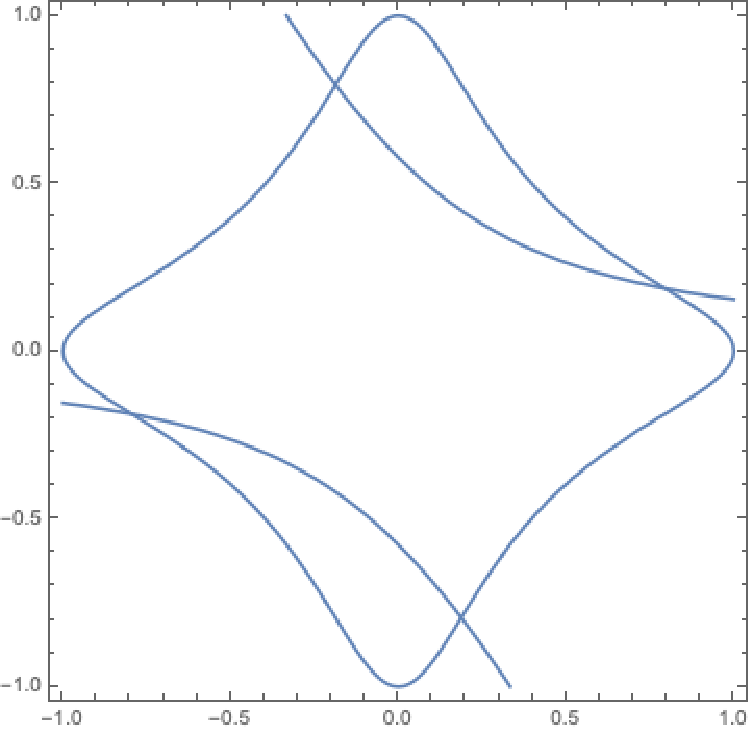}
	\label{fig:intersection}  
\end{figure}    
The two algebraic curves defined by the equations has one  intersection in $(0,1)\times (0,1)$.
Thus, there is Dziobek  equilibrium configuration on $\S^2$ that is not regular simplex but satisfies $\sum m_i\q_i=0$.


 \section*{Appendix: The derivative of the 
 	Cayley-Menger determinant}

 For a Dziobek configuration of $n$-body in $\S^{n-2}$, recall the $(n-1)\times n$ matrix  $X=[\q_1, ..., \q_n]$.  Since $\rank X =n-1$, the corresponding Gram matrix $X^TX$ has rank $n-1$. Then the determinant $F=0$. We may call the quantity  $F$  the \emph{spherical Cayley-Menger determinant}, \cite{Berger}.  For instance, for $n=4$, 
 \[   F=\begin{vmatrix}
 1 & \cos d_{12} & \cos d_{13} & \cos d_{14} \\
\cos d_{12} & 1 & \cos d_{23} & \cos d_{24} \\
\cos d_{13} & \cos d_{23} & 1 & \cos d_{34} \\
\cos d_{14} & \cos d_{24} & \cos d_{34} & 1 \\
 \end{vmatrix}.  \]  
 
 A by-product of equation \eqref{equ:SCCE2} is the following. A Dziobek configuration on $\S^{n-2}$ can be parametrized by  the $C_n^2$ quantities $\{\cos d_{12}, ..., \cos d_{n-1,n}\}$ with the relation $F=0$. Then any equilibrium configuration of the system \eqref{equ:Lag}  is the critical point of $V+\lambda F$.  
 Then equation \eqref{equ:SCCE2}  implies 
$\frac{\partial F}{\partial \cos d_{ij}}=\alpha \D_i\D_j$ for some $\alpha$. 

\begin{proposition}
	Let $\q_1, ..., \q_n$ be a Dziobek configuration in $\S^{n-2}$. Let $d_{12}, ..., d_{n-1, n}$ and $F$ be the corresponding mutual distances and the spherical Cayley-Menger determinant. Then we have 
	\[  \frac{\partial F}{\partial \cos d_{ij}}=2 \D_i\D_j \  {\rm for\ any}\ 1\le i<j\le n.  \]
where $\D_i$ is the signed determinant defined in \eqref{equ:D}. 
\end{proposition}

\begin{proof}
	By the symmetry of $X^TX$, we have $\frac{\partial F}{\partial \cos d_{ij}}=2 F_{ij}$, with
	$F_{ij}$ being the $(i, j)$ cofactor of matrix $X^TX$, i.e., 
	\[ \frac{\partial F}{\partial \cos d_{ij}}=2 (-1)^i (-1)^j  |A_{ij}|,  \]
	 where $A_{ij}$ is the  $(i, j)$ minor  of matrix $X^TX$. Let $X_k$ be the square matrix of order $n-1$ obtained from $X$ by deleting the $k$-th column. Then   $X_i^T X_j=A_{ij}$. 
	Thus, we have  $\frac{\partial F}{\partial \cos d_{ij}}=2 \D_i\D_j$. 
\end{proof}

This derivative formula  enables us to obtain equation \eqref{equ:SCCE2} directly.

For  a Dziobek configuration $\x=(\x_1, ..., \x_n )$ in $\R^{n-2}$, the mutual distances satisfy a relation and its  derivative formula is similar to  the above one.
Due to the translational symmetry, the appropriate Gram matrix is $\tilde X^T\tilde X$, with 
\[  \tilde X=[\x_2-\x_1, ..., \x_n-\x_1].  \]
 It is easy to see that $|\tilde X^T \tilde X|=0$. 
 Note that the entries of  $X^T\tilde X$ are not in terms of  the mutual distances. By using the formula $(\x_i-\x_1) \cdot (\x_j-\x_1)=\frac{1}{2}(d_{1i}^2 +d_{1j}^2 - d_{ij}^2)$ and some bordering technique, \cite{Berger}, we can obtain another determinant 
\[  \Gamma =\begin{vmatrix}
0&1&1&\cdots&1\\
1&0&d_{12}^2&\cdots&d_{1n}^2\\
1&d_{21}^2&0&\cdots&d_{2n}^2\\
\vdots&\vdots&\vdots&\ddots&\vdots&\\
1&d_{n1}^2&d_{n2}^2&\cdots&0
\end{vmatrix},  \ {\rm and}\  \Gamma=(-1)^n2^{n-1}|\tilde X^T\tilde X|. \]
Usually, it is $\Gamma$ instead of $|\tilde X^T\tilde X|$ that is called the \emph{Caylay-Menger determinant}.  
Let
\[ X=\begin{bmatrix}
1 & 1& \cdots & 1\\
\x_1&\x_2&\cdots&\x_n
\end{bmatrix} _{(n-1)\times n},  \]
and $X_k$ be the   square matrix of order $n-1$ obtained from $X$ by deleting the $k$-th column. Let $\D_k =(-1)^{k-1}|X_k|$. For $n=4$, Dziobek \cite{Dzi1900} observed  a formula that is  equivalent to 
\[  \frac{\partial \Gamma}{\partial d_{ij}^2}=-8 \D_i\D_j \  {\rm for\ any}\ 1\le i<j\le 4. \]
With the  technique used to relate $\Gamma$ and $|\tilde X^T \tilde X|$,  we have 
\begin{proposition}
	Let $\x_1, ..., \x_n$ be a Dziobek configuration in $\R^{n-2}$. Let $d_{12}, ..., d_{n-1, n}$  be the corresponding mutual distances. Let $\Gamma$ and $\D_i$ be the determinants defined above. Then we have 
	\[  \frac{\partial \Gamma}{\partial  d_{ij}^2}= (-2)^{n-1} \D_i\D_j \  {\rm for\ any}\ 1\le i<j\le n.  \]
\end{proposition}

\begin{proof}
	By the symmetry, we have $\frac{\partial \Gamma}{\partial d_{ij}^2}=2 (-1)^i (-1)^j  |B_{ij}|$，  
	where $B_{ij}$ is the  $(i+1, j+1)$ minor  of $\Gamma$. On the other hand,  note that 
	\[ | X_i|=\begin{vmatrix}
	1&0&0&\cdots &0\\
	0&1&1&\cdots &1\\
	\vec{0}&\x_1&\x_2&\cdots&\x_n
	\end{vmatrix} =- \begin{vmatrix}
	0&1&1&\cdots &1\\
	1&0&0&\cdots &0\\
	\vec{0}&\x_1&\x_2&\cdots&\x_n
	\end{vmatrix}.  \]
	Bordering $X_j$ in the same way without exchanging the first two row, we obtain 
	\[ |X_i^TX_j|=  - \begin{vmatrix} 0&1&\cdots &1&\cdots&1&1&\cdots&1\\
	1&\x_1 \cdot \x_1&\cdots&\x_1 \cdot \x_i&\cdots&\x_1 \cdot \x_{j-1}&\x_1 \cdot \x_{j+1}&\cdots&\x_1 \cdot \x_{n}\\
	\vdots&\vdots&\ddots&\vdots&\ddots&\vdots &\vdots&\ddots&\vdots&\\
		1&\x_{i-1} \cdot \x_1&\cdots&\x_{i-1} \cdot \x_i&\cdots&\x_{i-1}  \cdot \x_{j-1}&\x_{i-1}  \cdot \x_{j+1}&\cdots&\x_{i-1} \cdot \x_{n}\\
	1&\x_{i+1}  \cdot \x_1&\cdots&\x_{i+1}  \cdot \x_i&\cdots&\x_{i+1}  \cdot \x_{j-1}&\x_{i+1}  \cdot \x_{j+1}&\cdots&\x_{i+1}  \cdot \x_{n}\\
			\vdots&\vdots&\ddots&\vdots&\ddots&\vdots &\vdots&\ddots&\vdots&\\
	1&\x_j \cdot \x_1&\cdots&\x_j\cdot \x_i&\cdots&\x_j \cdot \x_{j-1}&\x_j \cdot \x_{j+1}&\cdots&\x_j \cdot \x_{n}\\
			\vdots&\vdots&\ddots&\vdots&\ddots&\vdots &\vdots&\ddots&\vdots&\\
				1&\x_n \cdot \x_1&\cdots&\x_n \cdot \x_i&\cdots&\x_n \cdot \x_{j-1}&\x_n \cdot \x_{j+1}&\cdots&\x_n \cdot \x_{n}
	\end{vmatrix} \]
	We then replace $\x_i\cdot \x_j$ be $\frac{1}{2}(||\x_i||^2 +||\x_j||^2 - d_{ij}^2)$, and eliminate all the $||\x_i||^2$ by subtracting the appropriate multiple of the first row and column from the others. We obtain 
	\[  \D_i\D_j= (-1)^{i+j}|X_i^T X_j| =(-1)^{i+j}  2^{2-n}   (-1)^{n-1}|B_{ij}|.   \]
	Hence follows the formula $\frac{\partial \Gamma}{\partial  d_{ij}^2}= (-2)^{n-1} \D_i\D_j$. 
\end{proof}

Central configuration  in $\R^n$ of dimension $n-2$ are considered in \cite{Moe01}. The  equations of them are derived by  vectorial method there. Note that these equations follows easily from the above derivative formula.


\end{document}